\documentclass[12pt]{amsart}
\usepackage{amsmath}
\usepackage{amsfonts}
\usepackage{amssymb}
\usepackage{amsthm}
\usepackage{url}
\usepackage{hyperref}
\usepackage{amsrefs}
\usepackage{indentfirst}
\newtheorem{theorem}{Theorem}[section]

\newtheorem{con}[theorem]{Conjecture}

\newtheorem{cor}[theorem]{Corollary}
\theoremstyle{remark}

\numberwithin{equation}{section}

\begin{document}
\title{Recurrence relations of Li coefficients}
\author{Huan Xiao}
\address{Department of Mathematics, Faculty of Science, Kumamoto University,
Kurokami 2-39-1 Kumamoto 860-8555 Japan}
\curraddr{}
\email{197d9003@st.kumamoto-u.ac.jp}
\thanks{}
\subjclass[2010]{}


\begin{abstract}
One of equivalents of the Riemann hypothesis is Li's criterion that all Li coefficients are positive. We study recurrence relations of Li coefficients in this note.
\end{abstract}

\maketitle

\section{Introduction}
The Riemann zeta function is the infinite series 
\begin{equation*}
\zeta(s)=\sum\limits_{n=1}^{\infty}\frac{1}{n^{s}}
\end{equation*}
where $ s:=\sigma+it $ is a complex number. It is well known that the Riemann zeta function has zeros at negative even integers which are called trivial zeros. The Riemann hypothesis asserts that all nontrivial zeros satisfy $ \sigma=\frac{1}{2} $. The Riemann zeta function is a vast subject in number theory, for its basic properties and other advanced progress one may refer to \cites{I,E,KV,T}. \par

Keiper \cite{Ke} studied the power series
\begin{center}
$ \log\left(2\xi\left(\dfrac{1}{s}\right)\right)=\sum\limits_{k=0}^{\infty}\lambda_{k}(1-s)^{k} $
\end{center}
where
\begin{center}
$ \xi(s)=\dfrac{s(s-1)}{2}\pi^{-\frac{s}{2}}\Gamma(\frac{s}{2})\zeta(s) $
\end{center}
is the Riemann Xi function and proved that the Riemann hypothesis implies $ \lambda_{k}>0 $ for all $ k>0 $. Later in 1997 Li \cite{Li} showed the converse, which is known as Li's criterion.
\begin{theorem}[Li's criterion]
The Riemann hypothesis is equivalent to the statement that all Li coefficients
\begin{center}
$ \lambda_{n}=\sum_{\rho}\left[1-\left(1-\dfrac{1}{\rho}\right)^{n}\right] $
\end{center}
are positive for integers $ n>0 $ where $ \rho $ runs over the nontrivial zeros of the Riemann zeta function.
\end{theorem}
The Li coefficients $ \lambda_{n} $ are extensively studied, see for example \cites{A,B,J,M}.\par 
In this note, we study recurrence relations of Li coefficients.
\subsection{Acknowledgement}
The author is partially supported by China Scholarship Council.

\section{Naive recurrence relation}
First we have the following naive recurrence relation of Li coefficients.
\begin{theorem}\label{thm1}
For integers $ n>1 $,
\begin{equation}
\lambda_{n}=\lambda_{n-1}+\sum_{\rho}\dfrac{1}{\rho}\left(1-\dfrac{1}{\rho}\right)^{n-1}.
\end{equation}
\end{theorem}
\begin{proof}
This follows from
\begin{equation}\label{eq0}
\lambda_{n}=\sum_{\rho}\left[1-\left(1-\dfrac{1}{\rho}\right)^{n}\right],
\end{equation}
and
\begin{equation}
1-\left(1-\dfrac{1}{\rho}\right)^{n}=\dfrac{1}{\rho}\left[1+\left(1-\dfrac{1}{\rho}\right)+\left(1-\dfrac{1}{\rho}\right)^{2}+\cdots +\left(1-\dfrac{1}{\rho}\right)^{n-1}\right].
\end{equation}
\end{proof}
So if one could show that
\begin{equation}
\sum_{\rho}\dfrac{1}{\rho}\left(1-\dfrac{1}{\rho}\right)^{n-1}>0
\end{equation}
for all $ n\geq 1 $, then the Riemann hypothesis is true. 

\section{Nontrivial recurrence relation}
Since $ \rho $ and $ 1-\rho $ is a pair of zeros, we have
\begin{equation}
2\lambda_{n}=\sum_{\rho}\left[1-\left(1-\dfrac{1}{\rho}\right)^{n}\right]+\sum_{\rho}\left[1-\left(1-\dfrac{1}{1-\rho}\right)^{n}\right],
\end{equation}
set 
\begin{equation}
t:=1-\dfrac{1}{\rho}=\dfrac{\rho-1}{\rho},
\end{equation}
then
\begin{equation}
t^{-1}=1-\dfrac{1}{1-\rho}=\dfrac{\rho}{\rho-1},
\end{equation}
therefore
\begin{equation}\label{eq2}
2\lambda_{n}=\sum_{\rho}\left[2-\left(t^{n}+t^{-n}\right)\right],
\end{equation}
\begin{equation}\label{eq3}
\lambda_{n}=\sum_{\rho}\left[1-\left(\dfrac{t^{n}+t^{-n}}{2}\right)\right].
\end{equation}
From (\ref{eq2}) we have
\begin{theorem}
For $ n\geq 1 $,
\begin{equation}
\lambda_{n}=-\dfrac{1}{2}\sum_{t}\left(t^{n/2}-t^{-n/2}\right)^{2}.
\end{equation}
\end{theorem}

We now relate  $ \lambda_{n} $ to the Chebyshev polynomials of the first kind $ T_{n}(x) $. Recall that $ T_{n}(x) $ is defined as
\begin{align*}
T_{0}(x)=1\\
T_{1}(x)=x\\
T_{n+1}(x)=2xT_{n}(x)-T_{n-1}(x).
\end{align*}
$ T_{n}(x) $ has the property that for $ x\neq 0 $,
\begin{equation}\label{eq5}
T_{n}\left(\dfrac{x+x^{-1}}{2}\right)=\dfrac{x^{n}+x^{-n}}{2},
\end{equation}
for more about Chebyshev polynomials see \cite{Ri}.\par
Now we take 
\begin{equation}
x=\dfrac{t+t^{-1}}{2}=1-\dfrac{1}{2\rho(1-\rho)},
\end{equation}
by (\ref{eq3}) and (\ref{eq5}) we have
\begin{equation}
\lambda_{n}=\sum_{\rho}\left[1-T_{n}\left(x\right)\right], 
\end{equation}
from the recurrence relation of $ T_{n}(x) $ we have
\begin{eqnarray*}
\lambda_{n}&=&\sum_{\rho}\left[1-\left(2xT_{n-1}(x)-T_{n-2}(x)\right)\right]\\
&=&\sum_{\rho}\left[1-\left(\left(2-\dfrac{1}{\rho(1-\rho)}\right)(T_{n-1}(x)-1)+\left(2-\dfrac{1}{\rho(1-\rho)}\right)-T_{n-2}(x)\right)\right]\\
&=& 2\lambda_{n-1}-\lambda_{n-2}+\sum_{\rho}\dfrac{1}{\rho(1-\rho)} T_{n-1}(x)\\
&=&-\dfrac{1}{2}\sum_{\rho}\left[\dfrac{1}{\rho^{2}}\left(1-\dfrac{1}{\rho}\right)^{n-2}+\dfrac{1}{(1-\rho)^{2}}\left(1-\dfrac{1}{1-\rho}\right)^{n-2}\right]+2\lambda_{n-1}-\lambda_{n-2}\\
&=&-\sum_{\rho}\dfrac{1}{\rho^{2}}\left(1-\dfrac{1}{\rho}\right)^{n-2}+2\lambda_{n-1}-\lambda_{n-2}.
\end{eqnarray*}

Thus we have the following recurrence relation of Li coefficients.
\begin{theorem}\label{thm2}
Let $ n\geq 3 $, one has
\begin{equation}
\lambda_{n}=-\sum_{\rho}\dfrac{1}{\rho^{2}}\left(1-\dfrac{1}{\rho}\right)^{n-2}+2\lambda_{n-1}-\lambda_{n-2}.
\end{equation}
\end{theorem}
The theorem above can be directly deduced by Theorem \ref{thm1}, but the relation between Li coefficients $ \lambda_{n} $ and Chebyshev polynomials is interesting.
\begin{cor}
Assume the Riemann hypothesis, for $ n\geq 3 $ we have
\begin{equation}
\left|\lambda_{n}+\lambda_{n-2}-2\lambda_{n-1}\right|\leq 2+\gamma-\log 4\pi=0.04619\cdots
\end{equation}
where $ \gamma=0.5772\cdots $ is the Euler constant.
\end{cor}
\begin{proof}
By Theorem \ref{thm2},
\begin{equation}\label{eq00}
\left|\lambda_{n}+\lambda_{n-2}-2\lambda_{n-1}\right|=\left|\sum_{\rho}\dfrac{1}{\rho^{2}}\left(1-\dfrac{1}{\rho}\right)^{n-2}\right|,
\end{equation}
since we have assumed Riemann hypothesis, so
\begin{equation}
\left|\left(1-\dfrac{1}{\rho}\right)\right|=1,\quad \sum_{\rho}\dfrac{1}{\left|\rho\right|^{2}}=\sum_{\rho}\dfrac{1}{\rho(1-\rho)}=2+\gamma-\log 4\pi,
\end{equation}
the value of $ \sum_{\rho}\dfrac{1}{\rho(1-\rho)} $ can be found in for example \cite{E}. From (\ref{eq00}) we have
\begin{equation}
\left|\lambda_{n}+\lambda_{n-2}-2\lambda_{n-1}\right|\leq\sum_{\rho}\left|\dfrac{1}{\rho^{2}}\left(1-\dfrac{1}{\rho}\right)^{n-2}\right|=\quad \sum_{\rho}\dfrac{1}{\left|\rho\right|^{2}}=2+\gamma-\log 4\pi.
\end{equation}
\end{proof}
So if one could find an $ n\geq 3 $ such that
\begin{equation*}
\left|\lambda_{n}+\lambda_{n-2}-2\lambda_{n-1}\right|> 2+\gamma-\log 4\pi,
\end{equation*}
then the Riemann hypothesis is false.\par

In view of Theorem \ref{thm2} we pose the following conjecture.
\begin{con}
For all $ n\geq 3 $, one has
\begin{equation}
\lambda_{n}>2\lambda_{n-1}-\lambda_{n-2}.
\end{equation}
\end{con}

\end{document}